\documentclass{amsart}
\usepackage{graphicx}
\usepackage{amsmath}
\usepackage{amsthm}
\usepackage{amssymb,bbm}%
\usepackage[numbers, square]{natbib}

\renewcommand{\Re}{\operatorname{Re}}
\renewcommand{\Im}{\operatorname{Im}}

\newcommand{\dd}{{\rm d}}
\newcommand{\E}{\mathbb E}
\newcommand{\bY}{\mathbf Y}

\newcommand{\R}{\mathbb{R}}
\newcommand{\Q}{\mathbb{Q}}
\newcommand{\N}{\mathbb{N}}

\newcommand{\C}{\mathbb{C}}
\newcommand{\Z}{\mathbb{Z}}

\renewcommand{\P}{\mathbb{P}}

\newcommand{\dist}{\mathop{\mathrm{dist}}}
\newcommand{\sinc}{\mathop{\mathrm{sinc}}}

\newcommand{\Var}{\mathop{\mathrm{Var}}\nolimits}

\newcommand{\Cov}{\mathop{\mathrm{Cov}}\nolimits}

\newcommand{\eee}{{\rm e}}

\newcommand{\Zeros}{{\mathbf{Zeros}}}

\newcommand{\MMM}{M}

\newcommand{\eps}{\varepsilon}

\newcommand{\todistr}{\overset{d}{\underset{n\to\infty}\longrightarrow}}
\newcommand{\todistrnon}{\overset{d}{\longrightarrow}}
\newcommand{\toweak}{\overset{w}{\underset{n\to\infty}\longrightarrow}}
\newcommand{\toweaknon}{\overset{w}{\longrightarrow}}
\newcommand{\tovague}{\overset{v}{\underset{n\to\infty}\longrightarrow}}
\newcommand{\tovaguenon}{\overset{v}{\longrightarrow}}

\newcommand{\toas}{\overset{a.s.}{\underset{n\to\infty}\longrightarrow}}
\newcommand{\ton}{\overset{}{\underset{n\to\infty}\longrightarrow}}

\newcommand{\ind}{\mathbbm{1}}

\newcommand{\iii}{\mathrm{i}}

\theoremstyle{plain}
\newtheorem{theorem}{Theorem}[section]
\newtheorem{lemma}[theorem]{Lemma}
\newtheorem{corollary}[theorem]{Corollary}

\theoremstyle{definition}

\theoremstyle{remark}
\newtheorem{remark}[theorem]{Remark}

\begin{document}

\title[Real roots of random trigonometric polynomials]{Local universality for real roots of random trigonometric polynomials}

\author{Alexander Iksanov}
\address{Alexander Iksanov, Faculty of Cybernetics, Taras Shevchenko National University of Kyiv, 01601 Kyiv, Ukraine}

\email{{iksan@univ.kiev.ua}}

\author{Zakhar Kabluchko}
\address{Zakhar Kabluchko,
Institut f\"ur Mathematische Statistik,
Westf\"{a}lische Wilhelms-Universit\"{a}t M\"{u}nster,
Orl\'eans--Ring 10,
48149 M\"unster, Germany}

\email{{zakhar.kabluchko@uni-muenster.de}}

\author{Alexander Marynych}
\address{Alexander Marynych, Faculty of Cybernetics, Taras Shevchenko National University of Kyiv, 01601 Kyiv, Ukraine}
\email{{marynych@unicyb.kiev.ua}}

\begin{abstract}
Consider a random trigonometric polynomial $X_n: \R \to \R$ of the form
$$
X_n(t) = \sum_{k=1}^n \left( \xi_k \sin (kt) + \eta_k \cos (kt)\right),
$$
where $(\xi_1,\eta_1),(\xi_2,\eta_2),\ldots$ are independent identically distributed bivariate real random vectors with zero mean and unit covariance matrix. Let $(s_n)_{n\in\N}$ be any sequence of real numbers. We prove that as $n\to\infty$, the number of real zeros of $X_n$ in the interval $[s_n+a/n, s_n+ b/n]$ converges in distribution to the
number of zeros in the interval $[a,b]$ of a stationary, zero-mean Gaussian process with correlation function $(\sin t)/t$. We also establish similar local universality results for the centered random vectors $(\xi_k,\eta_k)$ having an arbitrary covariance matrix or
belonging to the domain of attraction of a two-dimensional
$\alpha$-stable law.
\end{abstract}



\maketitle

\section{Introduction} \label{sec:Intro_and_main_results}

We are interested in random trigonometric polynomials $X_n: \R \to \R$ of the form
\begin{equation}\label{eq:def_X_n}
X_n(t) = \sum_{k=1}^n \left( \xi_k \sin (kt) + \eta_k \cos (kt)\right),
\end{equation}
where the coefficients $\xi_1,\eta_1, \xi_2,\eta_2,\ldots$ are real  random variables.
In a recent paper, Aza\"is et~al.~\cite{azais_etal} conjectured that if $\xi_1,\eta_1,\xi_2,\eta_2,\ldots$ are independent identically distributed (i.i.d.)\ with zero mean and finite variance, then the number of real zeros of $X_n$ in the interval $[a/n,b/n]$ converges in distribution (without normalization) to the
number of zeros in the interval $[a,b]$ of a stationary Gaussian process $Z:=(Z(t))_{t\in\R}$ with zero mean and
$$
\Cov (Z(t), Z(s)) = \sinc (t-s), \quad t,s\in\R,
$$
where
$$
\sinc t = \begin{cases}
(\sin t)/t, &\text{if } t\neq 0,\\
1, &\text{if } t=0.
\end{cases}
$$
The limit distribution does not depend on the distribution of $\xi_1$, a phenomenon referred to as local universality.
Aza\"is et~al.~\cite{azais_etal} proved their conjecture assuming that $\xi_1$ has an infinitely smooth density that satisfies certain  integrability conditions. However, as they remarked, even the case of the Rademacher distribution $\P[\xi_1=\pm 1]=1/2$  remained open. Our aim is to prove the conjecture of~\cite{azais_etal} in full generality (Theorem~\ref{theo:local_num_zeros} below). The method of proof proposed in the present paper is very different from the one used in~\cite{azais_etal}. Let us briefly sketch our approach assuming that $(\xi_1,\eta_1), (\xi_2,\eta_2),\ldots$ are i.i.d.\ random vectors such that $\xi_1$ and $\eta_1$ are centered uncorrelated random variables with unit variance.
We start by proving a functional limit theorem (Theorem~\ref{theo:FCLT} below) stating that
\begin{equation}\label{eq:fclt_basic}
\frac 1 {\sqrt n} X_n\left(s_n+\frac \cdot n\right) \ton  Z(\cdot)
\end{equation}
weakly on some suitable space of \textit{analytic} functions. Then, we use the continuous mapping theorem to deduce the convergence of the real zeros. The basic fact underlying this part of the proof is the Hurwitz theorem stating that the complex zeros of an analytic function do not change ``too much'' under a slight perturbation of the function. Essentially, Hurwitz's theorem tells us that the functional which maps an analytic function to the point process of its complex zeros is continuous. Since we are interested in real zeros, we have to  ensure that real zeros remain real after a small perturbation. If we restrict ourselves to analytic functions which are real on $\R$, then non-real zeros come in complex conjugated pairs, and a \textit{simple} real zero cannot become complex under a small perturbation of the function. These considerations, see Lemmas~\ref{lem:1} and~\ref{lem:1a}, justify the use of the continuous mapping theorem.

Our method is quite general and allows us to establish the corresponding local universality result in the case when $(\xi_1,\eta_1)$ has a non-zero correlation (Theorem~\ref{theo:local_arbitrary_cov_matrix}) or even does not have finite second moments but is in the domain of attraction of some stable two-dimensional law (Theorem~\ref{theo:local_stable_case}).

Closing the introduction, we mention that the scope of our approach is not restricted to trigonometric polynomials. The same method can be applied to various ensembles of random algebraic polynomials. A similar method was used in~\cite{starr}, \cite{ledoan_etal} for complex zeros of random Taylor series near the circle of convergence, in~\cite{shirai} for Dirichlet series with random coefficients and some other sums of analytic functions with random coefficients, and in~\cite{kabluchko_klimovsky1,kabluchko_klimovsky2} for complex zeros of the partition function of the (Generalized) Random Energy Model. Unlike in these works, we investigate real zeros.
Let us also mention that the asymptotics  of $\E N_n[a,b]$ (that is, the expected number of real zeros in an interval whose length does not go to $0$) was studied in the recent works~\cite{angst_poly} and~\cite{flasche}, where more references on random trigonometric polynomials can be found.

The structure of the paper is as follows. The main results are stated in Section~\ref{sec:main_results}. Functional limit theorems for $X_n$ and their proofs are given in Section~\ref{sec:convergence_polynomials}. In Section~\ref{sec:convergence_zeros} the proofs of the main theorems are presented. Some auxiliary technical lemmas are collected in the Appendix.

As usual, $\todistrnon$ denotes convergence in distribution of random  variables and vectors. The notation $\toweaknon$ is used to denote weak convergence of random elements with values in a metric space, while $\tovaguenon$ denotes vague convergence of locally finite measures.

\section{Main results}\label{sec:main_results}

\subsection{Coefficients with finite second moments} \label{subsec:coeff_finite_sec_moments}

For a real analytic function $f$ which does not vanish identically denote by $N_f[a,b]$ the number of zeros of $f$ in the interval $[a,b]$.
It will become clear from our proofs that the results hold independently of  whether the zeros are counted with multiplicities or not. Theorem~\ref{theo:local_num_zeros}, which is our first main result, proves the conjecture of~\cite{azais_etal}, weakens the original assumptions of~\cite{azais_etal} on the distribution of the coefficients and allows for an arbitrary sequence $(s_n)_{n\in\N}$ as the location of the scaling window.

\begin{theorem}\label{theo:local_num_zeros}
Let $(\xi_1,\eta_1),(\xi_2,\eta_2),\ldots$ be i.i.d.\ random vectors with zero mean and unit covariance matrix, that is,
$$
\E \xi_1= \E \eta_1 =0, \quad \E [\xi_1^2]=\E [\eta_1^2] =1, \quad \E [\xi_1\eta_1] = 0.
$$
Let $(s_n)_{n\in\N}$ be any sequence of real numbers and $[a,b]\subset \R$ a finite interval. Then,
$$
N_{X_n}\left[s_n+\frac an, s_n+\frac bn\right] \todistr N_Z[a,b],
$$
where $(Z(t))_{t\in\R}$ is the stationary Gaussian process defined in Section~\ref{sec:Intro_and_main_results}.
\end{theorem}
We can also prove the weak convergence of point processes of zeros. Given a locally compact metric space $\mathbb{X}$, denote by $\MMM_p(\mathbb{X})$ the space of locally finite point measures on $\mathbb{X}$ endowed with the vague topology. A random element with values in $\MMM_p(\mathbb{X})$ is called a point process on $\mathbb{X}$. We refer to~\cite{resnick_book} for the information on point processes and their weak convergence.
For a real analytic function $f$ which does not vanish identically denote by $\Zeros_\R(f)$ the locally finite point measure on $\R$ counting the real zeros of $f$ with multiplicities. The next theorem is stronger than Theorem~\ref{theo:local_num_zeros} in view of the continuous mapping theorem and Lemma~\ref{lem:1a} below.
\begin{theorem}\label{theo:local_num_zeros_point_proc}
Under the same assumptions as in Theorem~\ref{theo:local_num_zeros} we have
$$
\Zeros_{\R}\left(X_n\left(s_n + \frac{\cdot}{n}\right)\right) \toweak \Zeros_{\R}(Z(\cdot))
$$
on $\MMM_p(\R)$.
\end{theorem}

In the next theorem we consider i.i.d.\ random vectors with arbitrary covariance matrix. In particular, this theorem covers random trigonometric polynomials of the form $\sum_{k=1}^n \xi_k \sin (kt)$ and $\sum_{k=1}^n \eta_k \cos (kt)$ involving $\sin$ or $\cos$ terms only.
\begin{theorem}\label{theo:local_arbitrary_cov_matrix}
Let $(\xi_1,\eta_1),(\xi_2,\eta_2),\ldots$ be i.i.d.\ random vectors with
$$
\E \xi_k =  \E \eta_k = 0, \quad \E [\xi_k^2]= \sigma_1^2<\infty, \quad \E [\eta_k^2] =\sigma_2^2<\infty, \quad \E [\xi_k\eta_k] = \rho,
$$
where $0< \sigma_1^2 + \sigma_2^2 <\infty$. Then, for every fixed $s\in \R$,
$$
\Zeros_{\R}\left(X_n\left(s + \frac{\cdot}{n}\right)\right) \toweak \Zeros_{\R}(G(\cdot))
$$
on $\MMM_p(\R)$, where $(G(t))_{t\in\R}$ is a centered Gaussian process with covariance
\begin{align}
\lefteqn{\E[G(t_1)G(t_2)]} \label{eq:G_covar}\\
&=&
\begin{cases}
\frac {\sigma_1^2 + \sigma_2^2} 2 \sinc (t_1-t_2), & s\notin\pi\Z,\\
\frac {\sigma_1^2 + \sigma_2^2}2  \sinc (t_1-t_2)- \frac {\sigma_1^2 - \sigma_2^2}2 \sinc (t_1+t_2)+\rho \frac {1- \cos(t_1+t_2)}{t_1+t_2}, & s\in\pi\Z,
\end{cases}\notag
\end{align}
with the convention that $x\mapsto (1-\cos x)/x$ equals $0$ at $x=0$.
\end{theorem}
\begin{remark}
In Theorem~\ref{theo:local_arbitrary_cov_matrix} we can replace the fixed  $s$ by a general sequence $(s_n)_{n\in\N}$  as in Theorem~\ref{theo:local_num_zeros}, but then we have to replace the condition $s\notin \pi\Z$ with $\lim_{n\to\infty} n\dist (s_n, \pi \Z) =+\infty$  and $s\in \pi\Z$ with $\lim_{n\to\infty} n\dist (s_n, \pi \Z) = 0$. Here, we used the notation $\dist (s_n, \pi \Z) = \min\{|s_n - \pi k|\colon k\in\Z\}$.
\end{remark}

\subsection{Coefficients from a stable domain of attraction}
Let $(\xi_1,\eta_1), (\xi_2,\eta_2),\ldots$ be i.i.d.\ random vectors from the {\it strict} domain of attraction of a two-dimensional $\alpha$-stable distribution, $0<\alpha<2$. This means that there exist numbers $b_n>0$ such that
\begin{equation}\label{eq:domain_of_attraction}
\frac 1 {b_n} \left(\sum_{k=1}^n \xi_k, \sum_{k=1}^n \eta_k \right) \todistr \mathbf S_{\alpha,\nu},
\end{equation}
where $\mathbf S_{\alpha,\nu}$ is a non-degenerate two-dimensional $\alpha$-stable random vector with L\'evy measure $\nu$ and shift parameter $0$. The adjective ``strict'' is used to highlight that convergence~\eqref{eq:domain_of_attraction} holds without centering, in particular, it is assumed that $\E\xi_1=\E\eta_1=0$ if $\alpha>1$. We refer to~\cite{samorodnitsky_taqqu_book} for details on multivariate stable distributions and stable processes. Note that $\nu$ is a locally finite measure on $\R^2\setminus \{0\}$ which has the homogeneity property
$$
\nu(\lambda B) = \lambda^{-\alpha} \nu(B)
$$
for all $\lambda>0$ and all Borel sets $B\subset \R^2 \setminus \{0\}$.
In what follows we identify $\R^2$ and $\C$ via the canonical isomorphism and consider $\R^2$-valued processes as $\C$-valued and vice versa.

\begin{theorem}\label{theo:local_stable_case}
Assume that~\eqref{eq:domain_of_attraction} holds and let $s\in \R$ be fixed. Then
$$
\Zeros_{\R}\left(X_n\left(s + \frac{\cdot}{n}\right)\right) \toweak \Zeros_{\R}(Z_{\nu}(\cdot))
$$
on $\MMM_p(\R)$, where $(Z_{\nu}(t))_{t\in\R}$ is a stochastic process given by
\begin{equation} \label{eq:def_Z_nu}
Z_{\nu}(t)
=
\Im \int_0^{1} \eee^{\iii tu} {\rm d} L(u)
=
\int_0^{1} \sin (tu) {\rm d} \Re L(u)  + \int_0^{1} \cos(tu) {\rm d} \Im L(u),
\end{equation}
for $t\in\R$,
and $(L(u))_{u\in[0,1]}$ is a $\C$-valued $\alpha$-stable L\'evy process with zero drift, no Gaussian component, and the L\'{e}vy measure $\tilde \nu$ defined by
\begin{equation}\label{nu_tilde}
\tilde \nu(B) :=
\begin{cases}
\int_0^{1}\nu(\eee^{2\pi \iii y}B){\rm d}y, &\text{ if } s\notin \pi\Q,\\
\frac{1}{q}\sum_{k=1}^{q}\nu(\eee^{2\pi\iii k/q}B),&\text{ if } s=2\pi p/q,  \text{ with } p\in\Z, q\in \N \text{ coprime},
\end{cases}
\end{equation}
for all Borel sets $B\subset \C \setminus \{0\}$, with $\nu$ being the L\'{e}vy measure of $\mathbf S_{\alpha,\nu}$ in~\eqref{eq:domain_of_attraction}.
\end{theorem}
\begin{remark}\label{rem:partial_integration}
The integral in~\eqref{eq:def_Z_nu} (which need not exist in the Lebesgue--Stieltjes sense because $L$ has finite variation a.s.\ in the case $\alpha\in (0,1)$ only) is defined via integration by parts:
\begin{equation}\label{eq:def_Z_nu_partial}
\int_0^{1} \eee^{\iii tu} {\rm d} L(u) \overset{\rm{def}}{=} L(1)\eee^{\iii t}-\iii t \int_0^1 L(u)\eee^{\iii tu} {\rm d}u.
\end{equation}
See, e.g., \cite{iksanov_incr_response,samorodnitsky_taqqu_book} for the properties of such stochastic integrals.
\end{remark}
\begin{remark}
An interesting feature of Theorem~\ref{theo:local_stable_case} is that the behavior of the zeros near $s$ depends on whether $\tilde s:=s/(2\pi)$ is rational or not.  To see why such arithmetic effects show up, assume for a moment that $\xi_k$ and $\eta_k$ are independent and symmetric $\alpha$-stable. Then, $X_n(s)$ is also symmetric $\alpha$-stable with scaling parameter $\sigma_n$, where
\begin{align*}
\sigma_n^{\alpha}
=
\sum_{k=1}^n |\cos (ks)|^{\alpha} + \sum_{k=1}^n |\sin (ks)|^{\alpha}
=
\sum_{k=1}^n \left|\cos \left(2\pi \{k \tilde s\} \right)\right|^{\alpha} + \sum_{k=1}^n \left|\sin \left(2\pi \left\{k\tilde s\right\}\right)\right|^{\alpha}
\end{align*}
and $\{\cdot\}$ denotes the fractional part.
If $\tilde s$ is irrational, then the sequence $(\{k\tilde s\})_{k\in\N}$ is uniformly distributed on the interval $[0,1]$ by Weyl's equidistribution theorem, see, for example Theorem 2.1 and Example 2.1 in \cite{Kuipers+Niederreiter:1974}, whereas for rational $\tilde s = p/q$ it is uniformly distributed on the finite set $\{0,\frac 1q,\ldots, \frac{q-1}q\}$, whence
$$
\lim_{n\to\infty} \frac 1n \sigma_n^{\alpha} =
\begin{cases}
\int_{0}^1 (|\cos (2\pi u)|^{\alpha}+|\sin (2\pi u)|^{\alpha}) {\rm d} u, &\text{ if } s\notin \pi \Q,\\
\frac 1q \sum_{k=1}^q (|\cos (2\pi k/q)|^{\alpha}+|\sin (2\pi k/q)|^{\alpha}) , &\text{ if } s=2\pi p/q.
\end{cases}
$$
Note that for $\alpha=2$ (which corresponds to the finite variance case studied in Section~\ref{subsec:coeff_finite_sec_moments}), there is no difference between the rational and irrational cases because $\sin^2 t + \cos^2 t =1$.
\end{remark}
\begin{remark}
In Theorem~\ref{theo:local_stable_case} it is possible to replace the fixed $s$ by a sequence $(s_n)_{n\in\R}$ assuming that $s:= \lim_{n\to\infty} s_n$ exists and either $s\notin \pi\Q$ (the first case in~\eqref{nu_tilde}) or $s\in \pi \Q$ and $|s_n- s| = o(1/n)$ as $n\to\infty$ (the second case in~\eqref{nu_tilde}).
\end{remark}

\section{Convergence of random trigonometric polynomials as random analytic functions}\label{sec:convergence_polynomials}
\subsection{Spaces of analytic functions and analytic continuations of the processes $Z$, $G$ and $Z_{\nu}$}
Let $\mathcal H$ be the space of functions which are analytic on the entire complex plane.  We endow $\mathcal H$ with the topology of uniform convergence on compact sets. This topology is generated by the complete separable metric
$$
d(f,g) = \sum_{k\geq 1} \frac 1{2^k} \frac{\|f-g\|_{\bar {\mathbb D}_{k}}}{1+ \|f-g\|_{\bar {\mathbb D}_{k}}},
$$
where $\bar {\mathbb D}_r = \{|z|\leq r\}$ is the closed disk or radius $r>0$ around the origin, and $\|f\|_K=\sup_{z\in K} |f(z)|$ is the $\sup$-norm of $f$ on a compact set $K\subset \C$; see~\cite[pp.~151--152]{conway_book}. A random analytic function is a random element taking values in the space $\mathcal H$ endowed with the Borel $\sigma$-algebra. We refer to~\cite{ben_hough_book} and~\cite{shirai} for more information on random analytic functions.

Let $\mathcal H_{\R}$ be a closed subspace of $\mathcal H$ consisting of all functions $f\in \mathcal H$ which take real values on $\R$. Note that for every $f\in \mathcal H_{\R}$ we have $f(\bar z) = \overline {f(z)}$ for all $z\in\C$. Indeed, the functions $f(z)$ and $\overline {f(\bar z)}$ are analytic and coincide on $\R$. Hence, they must coincide everywhere on $\C$ by the uniqueness theorem for analytic functions. The space $\mathcal H_{\R}$ is endowed with the induced topology and metric.

Following the approach outlined in the introduction, we shall show that convergence~\eqref{eq:fclt_basic} and its counterparts in the case of correlated $(\xi_1,\eta_1)$ and in the stable case hold weakly on the space $\mathcal{H}_{\R}$. But, first of all, we have to construct analytic continuations of the limit processes $Z$, $G$ and $Z_{\nu}$ appearing in Theorems~\ref{theo:local_num_zeros}, Theorems~\ref{theo:local_arbitrary_cov_matrix} and~\ref{theo:local_stable_case}, respectively.

\subsubsection{The process $Z$}\label{subsec:proc_Z}
The stationary Gaussian process $(Z(t))_{t\in\R}$ appearing in Theorem~\ref{theo:local_num_zeros} can be extended analytically to the complex plane using the representation
\begin{equation}\label{eq:Gauss_proc_2}
Z(t) = \sum_{k\in\Z} \sinc (t- \pi k) N_k, \quad t\in\C,
\end{equation}
where $(N_k)_{k\in\Z}$ are i.i.d.\ real standard Gaussian random variables. The series in~\eqref{eq:Gauss_proc_2} converges uniformly on compact subsets of $\C$ because so does the series $\sum_{k\in\Z} |\sinc (t- \pi k)|^2$; see~\cite[Lemma~2.2.3]{ben_hough_book}. It follows that $(Z(t))_{t\in\C}$ is an analytic function on $\C$ with probability $1$.

The $\R^2$-valued process $((\Re Z(t), \Im Z(t)))_{t\in\C}$ is jointly real Gaussian in the sense that for all $t_1,\ldots,t_d\in\C$, the $2d$-dimensional random vector
$$
(\Re Z(t_1),\Im Z(t_1),\ldots, \Re Z(t_d), \Im Z(t_d))
$$
is real Gaussian. Clearly, $\E Z(t)=0$ for all $t\in\C$. The covariance structure of $(Z(t))_{t\in\C}$ is given by
\begin{align}
\E [Z(t) Z(s)] = \sinc (t-s), \quad t,s\in\C, \label{eq:cov_Z1} \\ 
\E [Z(t) \overline{Z(s)}] = \sinc (t-\bar s), \quad t,s\in\C.  \label{eq:cov_Z2} 
\end{align}
For instance, in the case when $t,s\notin \pi\Z $, we have
\begin{align*}
\E [Z(t)\overline{Z(s)}]
&=
\sum_{k\in\Z} \frac{\sin (t-\pi k)}{t-\pi k} \overline{\left(\frac{\sin (s-\pi k)}{s-\pi k}\right)}\\
&=
(\sin t) (\sin \bar s) \sum_{k\in\Z} \frac{1}{(t-\pi k)(\bar s-\pi k)}\\
&=
\frac {(\sin t) (\sin \bar s)}{\bar s-t} \sum_{k\in\Z} \left(\frac{1}{t-\pi k} - \frac 1{\bar s-\pi k}\right)\\
&=
\frac {(\sin t) (\sin \bar s)}{\bar s-t} (\cot t-\cot \bar s)\\
&=
\frac{\sin (t-\bar s)}{t-\bar s},
\end{align*}
where we used the partial fraction expansion of the cotangent. In the case when $t=\pi j$ for some $j\in\Z$, we have
$$
\E [Z(t) \overline{Z(s)}]
=
\sum_{k\in\Z} \sinc (t-\pi k) \overline{\sinc (s-\pi k)}
=
\overline{\sinc(s-\pi j)}
=
\sinc(t-\bar s)
$$
because $\sinc (t-\pi k)=1$ for $k=j$ and $0$ for $k\neq j$. The proof of~\eqref{eq:cov_Z1} is similar.

Representation~\eqref{eq:Gauss_proc_2} appeared, for example, in~\cite{antezana_etal}. Note that the analytically continued process $(Z(t))_{t\in\C}$ is stationary with respect to shifts along the real axis, but it is not stationary with respect to shifts along the imaginary axis.

\subsubsection{The process $G$} In the case $s\notin\pi\Z$ we can simply take
$$
G(t):=\sqrt{\frac {\sigma_1^2+\sigma_2^2} 2} Z(t), \quad t\in\C,
$$
where the process $(Z(t))_{t\in\C}$ is the same as in Section~\ref{subsec:proc_Z}. In the case $s\in\pi\Z$ take a centered $\C$-valued Brownian motion $(W(u))_{u\in [0,1]}$ with covariance structure
$$
\E [(\Re W(1))^2] = \sigma_1^2, \quad \E [(\Im W(1))^2] = \sigma_2^2, \quad \E [(\Im W(1))(\Re W(1))] = \rho,
$$
and put
\begin{align*}
U(t) = \int_0^1 \eee^{\iii tu}{\rm d}W(u), \quad t\in\C,
\end{align*}
where the integral is defined via the formal integration by parts, as in~\eqref{eq:def_Z_nu_partial}. Clearly, this defines $U$ as a random analytic function on $\C$. Now put\footnote{Here the following observation is used: if $f$ is an analytic function, so is $g(z):=(f(z)-\overline{f(\overline{z})})/(2\iii)$. Moreover, $g\in\mathcal{H}_{\R}$ and for $t\in\R$ we have $g(t)=\Im f(t)$. In particular, $G(t)=\Im U(t)$ for $t\in\R$, but, generally speaking, this relation fails for $t\in \C\setminus \R$. }
$$
G(t)=\frac{U(t)-\overline{U(\overline{t})}}{2\iii}=\int_0^1 \sin(tu){\rm d}\Re W(u)+\int_0^1 \cos(tu){\rm d}\Im W(u),\quad t\in\C.
$$
Integrating by parts and using the identities
\begin{align*}
&\int_{0}^1 \sin (t_1u) \sin (t_2u) {\rm d} u = \frac 12 \sinc (t_1-t_2) - \frac 12 \sinc (t_1+t_2),\\
&\int_{0}^1 \cos (t_1u) \cos (t_2u) {\rm d} u = \frac 12 \sinc (t_1-t_2) + \frac 12 \sinc (t_1+t_2),\\
&\int_{0}^1 \sin (t_1u) \cos (t_2u) {\rm d} u = \frac {1-\cos (t_1+t_2)}{2(t_1+t_2)}+ \frac {1-\cos (t_1-t_2)}{2(t_1-t_2)},
\end{align*}
it is easy to check that the covariance function of $(G(t))_{t\in\R}$ is given by the second line in~\eqref{eq:G_covar}.

\subsubsection{The process $Z_{\nu}$} Let $(L(u))_{u\in [0,1]}$ be a $\C$-valued $\alpha$-stable L\'evy process defined in Theorem~\ref{theo:local_stable_case}. As in the construction of $G$ above, put
\begin{equation}\label{eq:u_nu_def}
U_{\nu}(t)=\int_0^{1}\eee^{\iii tu}{\rm d}L(u),\quad t\in\C,
\end{equation}
where the integral is understood as in~\eqref{eq:def_Z_nu_partial}. Obviously, $U_{\nu}$ is a random analytic function on $\C$ and we can take
$$
Z_{\nu}(t)=\frac{U_{\nu}(t)-\overline{U_{\nu}(\overline{t})}}{2\iii}=\int_0^1 \sin(tu){\rm d}\Re L(u)+\int_0^1 \cos(tu){\rm d}\Im L(u),\quad t\in\C.
$$

\subsection{Functional limit theorems for random trigonometric polynomials}\label{subsec:FCLT_gauss}
Now we are ready to prove convergence~\eqref{eq:fclt_basic} and its counterparts corresponding to Theorems~\ref{theo:local_arbitrary_cov_matrix} and~\ref{theo:local_stable_case}.
\begin{theorem}\label{theo:FCLT}
Let $(\xi_1,\eta_1),(\xi_2,\eta_2),\ldots$ be  i.i.d.\ random vectors with zero mean and unit covariance matrix. Fix any sequence of real numbers $(s_n)_{n\in\N}$ and consider a random process $(Y_n(t))_{t\in\C}$ defined by
\begin{align}
Y_n(t)
&:= \frac 1 {\sqrt n} X_n\left(s_n + \frac tn\right)  \label{eq:def_Y_n}\\
&= \frac 1 {\sqrt n} \sum_{k=1}^n \left( \xi_k \sin \left( k\left(s_n + \frac tn\right)\right) + \eta_k \cos \left(k \left(s_n + \frac tn\right)\right)\right). \notag
\end{align}
Then $Y_n \toweak Z$ on $\mathcal H_{\R}$.
\end{theorem}
\begin{proof}
The proof consists of two steps.

\vspace*{2mm}
\noindent
\emph{Convergence of finite-dimensional distributions.}
Take any $t_1,\ldots,t_d\in\C$. We have the representation $Y_n(t) = V_{n,1}(t)+\ldots+ V_{n,n}(t)$, where
$$
V_{n,k}(t) := \frac 1{\sqrt n} \left( \xi_k \sin \left(k\left(s_n + \frac tn\right)\right) + \eta_k \cos \left(k \left(s_n + \frac tn\right)\right)\right).
$$
The $d$-dimensional complex random vector $\bY_n:=(Y_n(t_1),\ldots, Y_n(t_d))$ can be represented as a sum of independent, zero mean random vectors $(V_{n,k}(t_1), \ldots, V_{n,k}(t_d))$ over $k=1,\ldots,n$. To show that $\bY_n$ converges in distribution to $\bY := (Z(t_1),\ldots, Z(t_d))$, we  shall use the Lindeberg central limit theorem. First we need to check that for all $i,j=1,\ldots,d$,
\begin{align*}
&\lim_{n\to\infty} \E[Y_n(t_i)Y_n(t_j)] = \E[Z(t_i)Z(t_j)]=\sinc (t_i-t_j),\\
&\lim_{n\to\infty} \E[Y_n(t_i)\overline{Y_n(t_j)}] = \E[Z(t_i)\overline{Z(t_j)}]=\sinc (t_i-\bar t_j).
\end{align*}
It follows from~\eqref{eq:def_Y_n} that
$$
\E[Y_n(t_i)Y_n(t_j)] = \frac 1n \sum_{k=1}^n \cos \frac {k(t_i-t_j)} n \ton \int_{0}^1 \cos (u(t_i-t_j)) \dd u = \sinc (t_i-t_j)
$$
and, similarly,
$$
\E[Y_n(t_i)\overline{Y_n(t_j)}] = \frac 1n \sum_{k=1}^n \cos \frac {k(t_i-\bar t_j)} n \ton \int_{0}^1 \cos (u(t_i-\bar t_j)) \dd u = \sinc (t_i-\bar t_j).
$$
It remains to verify the Lindeberg condition. Take some $t\in\C$. For every $\eps>0$ we need to show that
$$
\lim_{n\to\infty}\sum_{k=1}^n \E \left[|V_{n,k}(t)|^2 \ind_{\{|V_{n,k}(t)| \geq \eps\}}  \right] = 0.
$$
Using the inequalities $|z_1+z_2|^2\leq 2|z_1|^2 +2 |z_2|^2$ and $|\sin z|\leq \cosh (\Im z)$, $|\cos z|\leq \cosh (\Im z)$, we obtain that  for all $k=1,\ldots,n$,
$$
|V_{n,k}(t)|^2 \leq \frac 2n (\cosh^2 (\Im t)) \cdot (\xi_k^2 +\eta_k^2).
$$
With  $C=2\cosh^2 (\Im t)$ we get
\begin{align*}
\sum_{k=1}^n \E \left[|V_{n,k}(t)|^2 \ind_{\{|V_{n,k}(t)| \geq \eps\}}  \right]
&\leq
\frac {C}n \sum_{k=1}^n \E \left[(\xi_k^2+\eta_k^2) \ind_{\{C(\xi_k^2+\eta_k^2)\geq n\eps^2\}}  \right]\\
&=
C\E \left[(\xi_1^2+\eta_1^2) \ind_{\{\xi_1^2+\eta_1^2\geq n\eps^2/C\}}\right]
\end{align*}
which converges to $0$ as $n\to\infty$ because $\E [\xi_1^2+\eta_1^2]<\infty$.

\vspace*{2mm}
\noindent
\emph{Tightness.}
In order to prove that the sequence $(Y_n)_{n\in\N}$ is tight on $\mathcal H$, it suffices to show that for every $R>0$,
\begin{equation}\label{eq:tightness_suff_cond}
\sup_{n\in\N} \sup_{|t|\leq R} \E |Y_n(t)|^2 < \infty,
\end{equation}
see~\cite[Lemma 4.2]{kabluchko_klimovsky1} or the remark after Lemma 2.6 in~\cite{shirai}.
For all $|t|\leq R$ and $n\in\N$ we have
$$
\E |Y_n(t)|^2 = \frac 1n \sum_{k=1}^n \cos \frac {k(t-\bar t)} n = \frac 1n \sum_{k=1}^n \cosh \frac {2k (\Im t)} n \leq  \cosh (2R) <\infty
$$
because $-R\leq \frac kn \Im t\leq R$ for all $k=1,\ldots,n$.

\vspace*{2mm}
It follows that $Y_n$ converges to $Z$ weakly on $\mathcal H$, as $n\to\infty$. Since $\mathcal H_{\R}$ is a closed subset of $\mathcal H$ and all processes under consideration have their sample paths in $\mathcal H_{\R}$, the convergence holds weakly on $\mathcal H_{\R}$, as well.
\end{proof}

The next theorem provides convergence of random trigonometric polynomials under the assumptions of Theorem~\ref{theo:local_arbitrary_cov_matrix}.

\begin{theorem}\label{theo:FCLT_arbitrary_cov_matrix}
Let $s\in \R$ be fixed and define a random process $(Y_n(t))_{t\in\C}$ by
$$
Y_n(t) := \frac 1 {\sqrt n} X_n\left(s + \frac tn\right).
$$
Under assumptions of Theorem~\ref{theo:local_arbitrary_cov_matrix} we have $Y_n\toweak G$ on $\mathcal H_\R$.
\end{theorem}
\begin{proof}
We use the same idea as in the proof Theorem~\ref{theo:FCLT}.  We have $Y_n(t)=\sum_{k=1}^n V_{n,k}(t)$ where
$$
V_{n,k}(t) :=  \frac 1 {\sqrt n} \left(\xi_k \sin \left( k\left(s + \frac tn\right)\right) + \eta_k \cos \left(k \left(s + \frac tn\right)\right)\right).
$$
We shall need  the standard trigonometric identities
\begin{equation}\label{eq:trig_sums}
\sum_{k=1}^n \cos (k\theta) = -\frac 12   + \frac{\sin \left((n+\frac 12) \theta\right)}{2\sin \frac \theta 2},
\quad
\sum_{k=1}^n \sin (k\theta) =
\frac 12 \cot \frac \theta 2   - \frac{\cos \left((n+\frac 12) \theta\right)}{2\sin \frac \theta 2}.
\end{equation}
where the case $\theta\in 2\pi\Z$ is understood by continuity. As in the proof of Theorem~\ref{theo:FCLT} the subsequent argument is divided into two steps.

\vspace*{2mm}
\noindent
\emph{Convergence of finite-dimensional distributions.}
First we prove that the covariances of $Y_n$ converge to those of $G$.
For all $t_1,t_2\in\C$, we have
\begin{align}
\E [Y_n(t_1) Y_n(t_2)] &=
\frac {\sigma_1^2} n \sum_{k=1}^n \sin\left(k \left(s+\frac{t_1}{n}\right)\right) \sin\left(k \left(s+\frac{t_2}{n}\right)\right) \label{eq:cov_Y_n_explicit}\\
&+
\frac {\sigma_2^2} n \sum_{k=1}^n \cos\left(k \left(s+\frac{t_1}{n}\right)\right) \cos\left(k \left(s+\frac{t_2}{n}\right)\right)\notag\\
&+
\frac {\rho} n \sum_{k=1}^n \sin \left(k \left(2s+\frac{t_1+t_2}{n}\right)\right).\notag
\end{align}
Denote the three terms on the right-hand side by $S_1(n), S_2(n), S_3(n)$.
Using the formula $2\sin x \sin y = \cos (x-y) - \cos (x+y)$  and then the first identity in~\eqref{eq:trig_sums} we obtain
\begin{align*}
S_1(n)
&=
\frac {\sigma_1^2} {2n} \sum_{k=1}^n \cos\left(k \frac{t_1-t_2}{n}\right)
-
\frac {\sigma_1^2} {2n} \sum_{k=1}^n \cos\left(k \left( 2s + \frac{t_1+t_2}{n}\right)\right) \\
&=
\frac {\sigma_1^2} 2 \sinc (t_1-t_2) - \frac {\sigma_1^2}{2n} \frac {\sin \left((2n+1) \left(s +\frac{t_1+t_2}{2n}\right) \right)}{2 \sin \left( s+ \frac{t_1+t_2}{2n}\right)} + o(1), \quad \text{as } n\to\infty.
\end{align*}
Sending $n\to\infty$ and considering the cases $\sin s \neq 0$ and $\sin s=0$ separately, we infer
$$
\lim_{n\to\infty} S_1(n) =
\begin{cases}
\frac {\sigma_1^2} 2 \sinc (t_1-t_2), &\text{ if } s\notin \pi\Z,\\
\frac {\sigma_1^2} 2 \sinc (t_1-t_2) - \frac {\sigma_1^2} 2 \sinc (t_1+t_2), &\text{ if } s\in \pi\Z.
\end{cases}
$$
Similarly, using the formula $2\cos x \cos y = \cos (x-y) + \cos (x+y)$ for the second sum we arrive at
$$
\lim_{n\to\infty} S_2(n) =
\begin{cases}
\frac {\sigma_2^2} 2 \sinc (t_1-t_2), &\text{ if } s\notin \pi\Z,\\
\frac {\sigma_2^2} 2 \sinc (t_1-t_2) + \frac {\sigma_2^2} 2 \sinc (t_1+t_2), &\text{ if } s\in \pi\Z.
\end{cases}
$$
Finally, in view of the second formula in~\eqref{eq:trig_sums},
\begin{align*}
S_3(n)
&=
\frac {\rho}{n} \left(\frac 12 \cot \left( s+ \frac{t_1+t_2}{2n}\right) - \frac {\cos \left((2n+1) \left(s +\frac{t_1+t_2}{2n}\right) \right)}{2 \sin \left( s+ \frac{t_1+t_2}{2n}\right)}\right).
\end{align*}
Sending $n\to\infty$ gives
$$
\lim_{n\to\infty} S_3(n) =
\begin{cases}
0, &\text{ if } s\notin \pi\Z,\\
\rho \frac {1- \cos(t_1+t_2)}{t_1+t_2}, &\text{ if } s\in \pi\Z.
\end{cases}
$$
Taking everything together and recalling the definition of the process $G$, see~\eqref{eq:G_covar}, we obtain
\begin{equation}\label{eq:cov_Y_n_converges1}
\lim_{n\to\infty} \E [Y_n(t_1) Y_n(t_2)] = \E [G(t_1) G(t_2)].
\end{equation}
Similar computation (with $t_2$ replaced by $\bar {t}_2$), yields
\begin{equation}\label{eq:cov_Y_n_converges2}
\lim_{n\to\infty} \E [Y_n(t_1) \overline{Y_n(t_2)}] = \E [G(t_1) \overline{G(t_2)}].
\end{equation}

Fix $t_1,\ldots,t_d\in\C$. In view of the convergence of the covariances established in~\eqref{eq:cov_Y_n_converges1} and~\eqref{eq:cov_Y_n_converges2}, to prove that $(Y_n(t_1),\ldots,Y_n(t_d))$ converges in distribution to $(G(t_1),\ldots,G(t_d))$ it is enough to verify the Lindeberg condition:
$$
\lim_{n\to\infty}\sum_{k=1}^n \E \left[|V_{n,k}(t)|^2 \ind_{\{|V_{n,k}(t)| \geq \eps\}}  \right] = 0,
$$
for every fixed $t\in\C$ and $\eps>0$. This can be done exactly in the same way as in the proof of Theorem~\ref{theo:FCLT}.

\vspace*{2mm}
\noindent
\emph{Tightness.} It is sufficient to check condition~\eqref{eq:tightness_suff_cond}. Starting with the equality $\E |Y_n(t)|^2 = \E [Y_n(t)Y_n(\bar t)]$ and applying~\eqref{eq:cov_Y_n_explicit} with $t_1=t=\bar t_2$, we arrive at
\begin{align*}
\E |Y_n(t)|^2
&=
\frac {\sigma_1^2} n \sum_{k=1}^n \left|\sin\left(k \left(s+\frac{t}{n}\right)\right)\right|^2
+
\frac {\sigma_2^2} n \sum_{k=1}^n \left|\cos\left(k \left(s+\frac{t}{n}\right)\right)\right|^2 \\
&+
\frac {\rho} n \sum_{k=1}^n \sin \left(k \left(2s+\frac{2\Re t}{n}\right)\right).
\end{align*}
Together with the inequalities $|\sin z|\leq \cosh (\Im z)$ and $|\cos z|\leq \cosh (\Im z)$, this implies condition~\eqref{eq:tightness_suff_cond}.
Combining pieces together, we see that $Y_n \to G$ weakly on $\mathcal H$ and hence, also on $\mathcal H_{\R}$.
\end{proof}

In the case of attraction to a stable law we have the following functional limit theorem. Since its proof is more involved than in the previous cases, it is given in the separate Section~\ref{subsec:FLT_stable_proof}.

\begin{theorem}\label{theo:FCLT_stable_case}
Fix $s\in\R$. Under the assumptions of Theorem~\ref{theo:local_stable_case},
$$
\left(\frac 1 {b_n} X_n\left(s+ \frac tn\right)\right)_{t\in\C}
\toweak
\left(
Z_{\nu}(t)
\right)_{t\in\C}
$$
on $\mathcal{H}_\R$.
\end{theorem}

\subsection{Proof of Theorem~\ref{theo:FCLT_stable_case}}\label{subsec:FLT_stable_proof}
We start with a well-known observation, see \cite{Resnick:1986}, that~\eqref{eq:domain_of_attraction} implies that the distribution of $(\xi_1,\eta_1)$ varies regularly in $\R^2$ with the limit measure $\nu$, which, in turn, is equivalent to the vague convergence
\begin{equation}\label{domain1}
n\P[b_n^{-1}(\xi+\iii \eta)\in\cdot]\tovague \nu(\cdot)
\end{equation}
on $\overline{\C}\setminus\{0\}$. Here, $\overline{\C}:=\C\cup \{+\infty\}$ denotes the Riemann sphere, and $\overline{\C}\setminus\{0\}$ is the Riemann sphere with the punctured south pole. These spaces can be identified with  $\overline{\R^2}:= \R^2\cup\{\infty\}$ (the one-point compactification of $\R^{2}$) and $\overline{\R^2}\setminus \{0\}$, respectively. The measure $\nu$ is thought of as a measure on $\overline{\C}\setminus\{0\}$ by setting $\nu(\{\infty\})=0$.

The proof of Theorem~\ref{theo:FCLT_stable_case} is presented in the series of lemmas.

\begin{lemma}\label{point_processes_convergence}
Fix $s\in\R$ and define a sequence of point processes on $[0,\infty)\times (\overline{\C}\setminus\{0\})$ as follows:
$$
N_n:=\sum_{k\geq 1}\delta_{\left(\frac{k}{n},\frac{\xi_k+\iii\eta_k}{b_n}\eee^{\iii ks}\right)},\quad n\in\N.
$$
Then
\begin{equation}\label{eq:convergence_point_proc_stable_case}
N_n\tovague N_{\infty}
\end{equation}
on $M_p([0,\infty)\times (\overline{\C}\setminus\{0\}))$, where $N_{\infty}$ is a Poisson point process on $[0,\infty)\times (\overline{\C}\setminus\{0\})$ with intensity measure $\mathbb{LEB}\times \tilde \nu$, and $\tilde \nu$ is as in~\eqref{nu_tilde}.
\end{lemma}
\begin{proof}
Define the sequence $(\lambda_n)_{n\in\N}$ of measures on $[0,\infty)\times (\overline{\C}\setminus\{0\})$ as follows:
$$
\lambda_n({\rm d}x,{\rm d}z):=\sum_{k\geq 1}\delta_{k/n}({\rm d}x)\P[b_n^{-1}(\xi_k+\iii\eta_k)\eee^{\iii ks}\in {\rm d}z],\quad n\in\N,
$$
and let us show that
\begin{equation*}
\lambda_n\tovague
\mathbb{LEB}\times \tilde \nu
\end{equation*}
on $[0,\infty)\times (\overline{\C}\setminus\{0\})$. To this end, fix a continuous function $f:[0,\infty)\times (\overline{\C}\setminus\{0\})\to\R^{+}$ with compact support and pick $a,r>0$ such that $f(x, z)=0$ if $x>a$ or $|z|<r$.

\vspace*{2mm}
\noindent
{\sc Case $s\notin \pi\Q $.} We have to check that
$$
\sum_{k\geq 1}\int_{|z|\geq r}f(k/n,z)\P[b_n^{-1}(\xi_k+\iii\eta_k)\eee^{\iii ks}\in {\rm d}z]\ton \int_{|z|\geq r}\int_{0}^{a}f(x,z){\rm d}x\tilde \nu({\rm d}z).
$$
The left-hand side of the latter relation equals
$$
\int_{|z|\geq r}\left(\frac{1}{n}\sum_{k\geq 1}f(k/n,\eee^{\iii ks}z)\right)\left(n\P[b_n^{-1}(\xi_1+\iii\eta_1)\in {\rm d}z]\right),
$$
and, in view of equation~\eqref{weyl_uniform_convergence_irrational} in Lemma~\ref{weyl_uniform_lemma} in the Appendix and~\eqref{domain1}, converges to
\begin{align*}
\int_{|z|\geq r}\int_0^a \int_0^{1}f(x,\eee^{2\pi \iii y} z){\rm d}y{\rm d}x\nu({\rm d}z)
&=\int_0^{1}\int_{|z|\geq r}\int_0^a f(x, z){\rm d}x\nu(\eee^{-2\pi \iii y}{\rm d}z){\rm d}y\\
&=\int_{|z|\geq r}\int_0^a f(x,z){\rm d}x\tilde \nu({\rm d}z).
\end{align*}

\vspace*{2mm}
\noindent
{\sc Case $s=2\pi p/q$} follows analogously from~\eqref{weyl_uniform_convergence_rational} in Lemma~\ref{weyl_uniform_lemma}.

\vspace*{2mm}
The rest of the proof mimics the proof of Proposition~3.1 in \cite{Resnick:1986}. The only place which has to be checked is relation~(3.3) of the cited paper, which in our situation reads
$$
\lim_{n\to\infty}\sup_{k\geq 1}\P[b_n^{-1}(\xi_1+\iii\eta_1)\eee^{\iii ks}\in A]=0,
$$
where $A$ is a compact subset of $\overline{\C}\setminus\{0\}$. But this is obvious, since, by~\eqref{domain1},
$$
\sup_{k\geq 1}\P[b_n^{-1}(\xi_1+\iii\eta_1)\eee^{\iii ks}\in A]\leq \P[b_n^{-1}|\xi_1+\iii\eta_1|\in \{|z|\colon z\in A\} ]\ton 0.
$$
The proof of Lemma~\ref{point_processes_convergence} is complete.
\end{proof}

In what follows $D([0,\,1],\C)$ is the Skorokhod space of complex-valued functions defined on the interval $[0,\,1]$ which are right-continuous on $[0,\,1)$ and have finite limits from the left on $(0,\,1]$. The space $D([0,\,1],\C)$ is endowed with the usual $J_1$-topology; see~\cite{billingsley_book}.

\begin{lemma}\label{processes_converge}
Fix $s\in\R$ and define a sequence of $\C$-valued processes
\begin{equation}\label{eq:L_n_def}
L_n(t):=\frac 1 {b_n}\sum_{k=1}^{[nt]}(\xi_k+\iii\eta_k)\eee^{\iii ks}, \quad t\in[0,\,1].
\end{equation}
Then,
\begin{equation}\label{irrational_convergence_processes}
(L_n(t))_{t\in[0,\,1]}\toweak (L(t))_{t\in[0,\,1]}
\end{equation}
on $D([0,\,1],\C)$, where the L\'{e}vy process $L$ is the same as in Theorem~\ref{theo:local_stable_case}.
\end{lemma}
\begin{proof}
If $s\in 2\pi\Z$, then~\eqref{irrational_convergence_processes} is just a functional limit theorem for i.i.d.\ vectors corresponding to~\eqref{eq:domain_of_attraction}.
Let us assume that $s\notin 2\pi\Z $, which means that the vectors are independent but not identically distributed.  We shall use a criterion for functional convergence given in Theorem 3.1 in \cite{Turan-Kaminska:2010}. In view of Lemma~\ref{point_processes_convergence} we need to check that, for every $\delta>0$,
\begin{equation}\label{doublelimit}
\lim_{\varepsilon\to 0}\limsup_{n\to\infty}\P\left[\sup_{0\leq t\leq 1}
\left|b_n^{-1}\sum_{k=1}^{[nt]}(\xi_k+\iii\eta_k)\eee^{\iii k s}\ind_{\{|\xi_k+\iii\eta_k|\leq b_n\varepsilon\}}+t\int_{\varepsilon< |z|\leq 1 }z\tilde \nu({\rm d}z)\right|\geq \delta\right]=0.
\end{equation}
It follows from the definition of $\tilde \nu$ that it is invariant under the transformations $z\mapsto z \eee^{2\pi \iii \theta}$, where $\theta\in\R$ (if $s\notin \pi\Q$) and $\theta\in q^{-1}\Z$ (if $s = 2\pi p/q$).
Since we assume $s\notin 2\pi\Z$,  this transformation group contains at least one non-trivial rotation which implies that
\begin{equation}\label{rotation_invariant}
\int_{\{\varepsilon< |z|\leq 1 \}}z\tilde \nu({\rm d}z)
=0.
\end{equation}
The next step is to show that
\begin{equation}\label{doublelimit1}
\lim_{\varepsilon\to 0}\limsup_{n\to\infty}\P\left[\sup_{1\leq m\leq n}\left|\sum_{k=1}^{m}\Delta_{n,k}\right|\geq b_n\delta\right]=0,
\end{equation}
where
$$
\Delta_{n,k}=\Delta_{n,k}(s):=\left((\xi_k+\iii\eta_k)\ind_{\{|\xi_k+\iii\eta_k|\leq b_n\varepsilon\}}-\E[(\xi_k+\iii\eta_k)\ind_{\{|\xi_k+\iii\eta_k|\leq b_n\varepsilon\}}]\right)\eee^{\iii k s}.
$$
Note that $\E [\Delta_{n,k}]=0$. Since $\left(\left|\sum_{k=1}^{m}\Delta_{n,k}\right|\right)_{m\in\N}$ is a non-negative submartingale, we can apply Doob's inequality:
$$
\P\left[\sup_{1\leq m\leq n}\left|\sum_{k=1}^{m}\Delta_{n,k}\right|\geq b_n\delta\right]\leq (\delta b_n)^{-2}\E\left|\sum_{k=1}^{n}\Delta_{n,k}\right|^2.
$$
Further,
\begin{align*}
(\delta b_n)^{-2}\E\left|\sum_{k=1}^{n}\Delta_{n,k}\right|^2
&=(\delta b_n)^{-2}\sum_{k=1}^{n}\Var \Delta_{n,k}\\
&\leq (\delta b_n)^{-2}n \E \left[(\xi^2+\eta^2)\ind_{\{\sqrt{\xi^2+\eta^2}\leq b_n\varepsilon\}}\right],
\end{align*}
where $(\xi,\eta)$ is a distributional copy of $(\xi_1,\eta_1)$.
Assumption~\eqref{domain1} implies that $x\mapsto \P[\sqrt{\xi^2+\eta^2}>x]$ is regularly varying. Hence by Karamata's theorem in the form given by formula~(5.22) on p.~579 in \cite{Feller:1971},
$$
(\delta^2 b_n^{-2})n\E\left[(\xi^2+\eta^2)\ind_{\{\sqrt{\xi^2+\eta^2}\leq b_n\varepsilon\}} \right]\sim c\delta^{-2}\varepsilon^2 n\P[\sqrt{\xi^2+\eta^2}>\varepsilon b_n],\quad n\to\infty,
$$
for some $c>0$. Therefore,
\begin{align*}
\limsup_{n\to\infty}\P\left[\sup_{1\leq m\leq n}\left|\sum_{k=1}^{m}\Delta_{n,k}\right|\geq b_n\delta\right]
&\leq c\delta^{-2}\varepsilon^2\lim_{n\to\infty} n\P[\sqrt{\xi^2+\eta^2}>\varepsilon b_n]\\
&=    c\delta^{-2}\varepsilon^2 \int_{|z|>\varepsilon}\tilde \nu({\rm d}z).
\end{align*}
The last expression tends to zero, as $\varepsilon\to 0$, since $\tilde \nu$ is a L\'{e}vy measure, whence~\eqref{doublelimit1}.

Combining~\eqref{rotation_invariant}, \eqref{doublelimit1} and the trivial bound
\begin{align*}
\sup_{0\leq t\leq 1}\left|\E\left[(\xi+\iii\eta)\ind_{\{|\xi+\iii\eta|\leq b_n\varepsilon\}}\right]\sum_{k=1}^{[nt]}\eee^{\iii k s}\right|\leq \frac{2\varepsilon b_n}{|1-\eee^{\iii s}|},
\end{align*}
we see that~\eqref{doublelimit} holds.
\end{proof}

\begin{lemma}\label{main_convergence_lemma}
Fix $s\in\R$ and define a sequence of processes
\begin{equation}\label{eq:Y_n_def}
Y_n(t)
:=
\frac 1 {b_n}\sum_{k=1}^{n}(\xi_k+\iii\eta_k)\exp{\left(\iii k\left(s+\frac{t}{n}\right)\right)}, \quad t\in\C.
\end{equation}
Under the assumptions of Theorem~\ref{theo:local_stable_case} we have
\begin{equation}\label{irrational_convergence_polynomials}
(Y_n(t))_{t\in \C} \toweak \left(U_{\nu}(t)\right)_{t\in\C}
\end{equation}
on $\mathcal{H}$, where the process $U_{\nu}$ is defined in~\eqref{eq:u_nu_def}.
\end{lemma}
\begin{proof}
Define a mapping $\mathcal{F}:D([0,\,1],\C)\to\mathcal{H}$ as follows:
\begin{equation}\label{mapping_definition}
(\mathcal{F}(f))(z):=\int_{[0,\,1]}\eee^{\iii z x}{\rm d}f(x)\overset{\rm{def}}{=}f(1)\eee^{\iii z}-f(0)-\iii z\int_0^1 f(x)\eee^{\iii z x}{\rm d}x,\quad z\in\C.
\end{equation}
Since $f\in D([0,\,1],\C)$ ensures  $\sup_{t\in [0,\,1]}|f(t)|<\infty$, the function $\mathcal{F}(f)$ is analytic on the entire complex plane. Thus, $\mathcal{F}$ is indeed a well-defined mapping from $D([0,\,1],\C)$ to $\mathcal{H}$. By Lemma~\ref{continuous_mapping} in the Appendix the mapping $\mathcal{F}$ is everywhere continuous on $D([0,\,1],\C)$.

In view of the representation $Y_n=\mathcal{F}(L_n)$, with $L_n$ as in~\eqref{eq:L_n_def}, convergence~\eqref{irrational_convergence_polynomials} follows from the  continuous mapping theorem.
\end{proof}

Now we are in a position to prove Theorem~\ref{theo:FCLT_stable_case}.
\begin{proof}[Proof of Theorem~\ref{theo:FCLT_stable_case}]
Recalling the definition of $X_n$, we can write
\begin{align*}
\frac 1 {b_n} X_n\left(s+ \frac tn\right)
&=
\frac 1 {b_n} \sum_{k=1}^{n} \left(\xi_k\sin{\left(k\left(s+\frac{t}{n}\right)\right)}+\eta_k\cos\left(k\left(s+\frac{t}{n}\right)\right)\right)\\
&=
\frac{Y_n(t)-\overline{Y_n(\overline{t})}}{2i}
\end{align*}
with $Y_n$ as in~\eqref{eq:Y_n_def}.
It follows from Lemma~\ref{main_convergence_lemma} that
\begin{align*}
\left(\frac 1 {b_n} X_n\left(s+ \frac t n\right)\right)_{t\in\C}
&\toweak
\left(\frac{U_{\nu}(t)-\overline{U_{\nu}(\overline{t})}}{2i} \right)_{t\in\C}\\
&=
\left(
\int_0^{1} \sin (tu) {\rm d} \Re L(u)  + \int_0^{1} \cos(tu) {\rm d} \Im L(u)
\right)_{t\in\C}
\end{align*}
on $\mathcal {H}$. Since the processes under consideration have their sample paths in $\mathcal H_{\R}$ (which is a closed subset of $\mathcal H$), the convergence holds weakly on $\mathcal H_\R$, too.
\end{proof}

\section{Convergence of zeros}\label{sec:convergence_zeros}
Take some interval $[a,b]\subset \R$ and consider a mapping $N:\mathcal H_{\R}\setminus\{0\}\to \{0,1,\ldots\}$ which assigns to each function $f\in \mathcal H_{\R}$ the number of real zeros of $f$ in the interval  $[a,b]$. Although this will be irrelevant, let us agree that the zeros are counted with multiplicities.

\begin{lemma}\label{lem:1}
Let $A=A[a,b]\subset \mathcal H_{\R}$ be the set consisting of all $f\in \mathcal H_{\R}$ which do not have multiple real zeros in $[a,b]$ and satisfy  $f(a)\neq 0$, $f(b)\neq 0$. Then, the set $A$ is open and the mapping $N$ is locally constant on $A$ (that is, for every $f\in A$ there is an open neighborhood of $f$ in $\mathcal H_{\R}$ on which $N$ is constant).
\end{lemma}
\begin{proof}
Consider any sequence $(f_n)_{n\in\N}\subset \mathcal H_{\R}$ which converges to $f\in A$ locally uniformly. We need to show that for sufficiently large $n$ we have $f_n\in A$ and $N(f_n)=N(f)$. Let $R>0$ be so large that $[a,b]$ is contained in the open disk $\mathbb D_R=\{|z|<R\}$. Let $z_1,\ldots, z_d$ be the collection of all zeros of $f$ in $\mathbb D_R$ with corresponding multiplicities $m_1,\ldots,m_d$.
Assume without loss of generality that $f$ has no zeros on the boundary of $\mathbb D_R$ (just increase $R$, otherwise). Let $\eps>0$ be so small that the open $\eps$-disks $z_1+\mathbb D_\eps,\ldots, z_d+\mathbb D_\eps$ do not intersect each other, the boundary of $\mathbb D_R$, and the real axis (except when the zero is itself real).  By Hurwitz's theorem~\cite[p.~152]{conway_book}, for all sufficiently large $n$, the function $f_n$ has exactly $m_k$ zeros (with multiplicities) in the disk $z_k+\mathbb D_\eps$, for all $k=1,\ldots, d$,  and there are no other zeros of $f_n$ in $\mathbb D_R$. If $z_k\in (a,b)$, then $m_k=1$ (in view of $f\in A$) and the corresponding zero of $f_n$ in the disk $z_k+\mathbb D_\eps$ is also real because otherwise $f_n$ would have two different complex conjugated zeros (recall that $f_n(\bar z) = \overline{f_n(z)}$), which is a contradiction. It follows that all real zeros of $f_n$ in $(a,b)$ are simple and their number is $N(f)$. Clearly, $f_n(a)\neq 0$ and $f_n(b)\neq 0$ for sufficiently large $n$.  Hence, $f_n\in A$ and $N(f_n)=N(f)$ for large $n$.
\end{proof}

Recall that $\Zeros_{\R}(f)$ is a locally finite measure on $\R$ counting the real zeros of $f\in \mathcal H_\R\setminus\{0\}$ with multiplicities.
\begin{lemma}\label{lem:1a}
Let $A(\R)$ be the set of all $f\in \mathcal H_{\R}$ which do not have multiple real zeros. Consider a mapping $f\mapsto \Zeros_{\R}(f)$ from $\mathcal H_{\R}\setminus\{0\}$ to the space $\MMM_p(\R)$ of locally finite point measures on $\R$ endowed with the vague topology. Then, this mapping is continuous on $A(\R)$.
\end{lemma}
\begin{proof}
Let $(f_n)_{n\in\N}\subset \mathcal H_{\R}$ be a sequence which converges to $f\in A(\R)$ locally uniformly. Fix $R>0$. Let $z_1,\ldots, z_l$ be the real zeros of $f$ in $[-R,R]$ and assume there are no zeros at $-R$ and $R$. Fix $\eps>0$. Arguing as in the proof of Lemma~\ref{lem:1}, we can show that for sufficiently large $n$, the function $f_n$ has exactly one real zero in any of the disks $z_1+\mathbb D_\eps,\ldots, z_l+\mathbb D_\eps$ and there are no further real zeros of $f_n$ in $[-R,R]$. But this means that $\Zeros_{\R}(f_n)$ converges to $\Zeros_{\R}(f)$ vaguely.
\end{proof}
\subsection{Proofs of Theorems~\ref{theo:local_num_zeros}, \ref{theo:local_num_zeros_point_proc}, \ref{theo:local_arbitrary_cov_matrix} and~\ref{theo:local_stable_case}}
In view of the last two lemmas, Theorems~\ref{theo:FCLT}, \ref{theo:FCLT_arbitrary_cov_matrix} and~\ref{theo:FCLT_stable_case} and the continuous mapping theorem, convergence of zeros in Theorems~\ref{theo:local_num_zeros_point_proc}, \ref{theo:local_arbitrary_cov_matrix} and~\ref{theo:local_stable_case} follows,
if we can show  that
\begin{equation}\label{no_multiple_zeros}
\P[Z\in A(\R)]=\P[G\in A(\R)]=\P[Z_{\nu}\in A(\R)]=1.
\end{equation}
Analogously, Theorem~\ref{theo:local_num_zeros} is a consequence of
\begin{equation}\label{no_multiple_zeros_in_interval}
\P[Z\in A([a,b])]=1,
\end{equation}
for every $a<b$. In order to verify these statements, we need the following result due to E.\ V.\ Bulinskaya~\cite{bulinskaya}. It provides general conditions which ensure that a stochastic process (which need not be Gaussian) does not have multiple zeros, with probability $1$.
\begin{lemma}\label{lem:no_mult_zeros_general}
Let $(Q(t))_{t\in [a,b]}$  be a stochastic process with continuously differentiable sample paths. Assume that the random variables $Q(t)$ are absolutely continuous with densities which are bounded uniformly in $t\in [a,b]$. Then, with probability $1$ there is no $t\in[a,b]$ such that $Q(t)=Q'(t)=0$.
\end{lemma}

The parts of~\eqref{no_multiple_zeros} and~\eqref{no_multiple_zeros_in_interval} regarding the Gaussian processes $Z$ and $G$ (in the case $s\notin \pi \Z$) follow immediately from Lemma~\ref{lem:no_mult_zeros_general}   (see also~\cite{ylvisaker} or~\cite{ylvisaker_stat} for further work on the absence of multiple zeros of Gaussian processes). Indeed, the variances of both $Z$ and $G$ are non-zero constants which
implies that there are uniform upper bounds on the densities of $Z(t)$ and
$G(t)$.
 Let us consider $G(t)$ in the case $s\in \pi \Z$.
\begin{lemma}
Let $s\in \pi \Z$. Then, with probability $1$ there is no $t\in\R$ such that $G(t)= G'(t)=0$.
\end{lemma}
\begin{proof}
We consider only $t\geq 0$ because the case $t < 0$ is similar.
For all $t>0$ we have
$$
\Var G(t)= \frac {\sigma_1^2 + \sigma_2^2}2 - \frac {\sigma_1^2 - \sigma_2^2}2 \sinc (2t)+\rho \frac {1- \cos(2t)}{2t}.
$$
The function $t\mapsto t\Var G(t)$ is non-decreasing since
$$
\frac {{\rm d}}{{\rm d} t}(2t\Var G(t)) = (\sigma_1^2 + \sigma_2^2) - (\sigma_1^2 - \sigma_2^2) \cos (2t) + 2\rho \sin (2t) \geq 0,
$$
where in the latter inequality we used that $|\rho| \leq \sigma_1 \sigma_2$ (Cauchy--Schwarz inequality) and the maximum of the function $a \cos (2t) + b \sin (2t)$ is $\sqrt{a^2+b^2}$.

\vspace*{2mm}
\noindent {\sc Case $\Var G(0)=\sigma_2^2>0$}.  We can find $\eps>0$ such that $\Var G(t) > \frac12 \sigma_2^2>0$ for all $0\leq t\leq\eps$. For $t\geq \eps$, we obtain $\Var G(t) \geq  \eps \sigma_2^2 /(2t)$ and consequently $\Var G(t)$ is bounded below on compact sets, thus justifying the use of Lemma~\ref{lem:no_mult_zeros_general}.

\vspace*{2mm}
\noindent {\sc Case $\Var G(0)=\sigma_2^2=0$}. Then $\rho=0$ and $\Var G(t) = \frac 12 \sigma_1^2 (1-\sinc (2t))$ is still uniformly bounded away from zero on $[a,b]$ for all
$0<a<b<\infty$.  By Lemma~\ref{lem:no_mult_zeros_general} there are no multiple zeros of $G$ on $(0,\infty)$, with probability $1$. To see that $0$ is a.s.\ not a multiple zero, note that although $G(0)=0$, we have
$G'(0) = \int_0^1 u {\rm d}\Re W(u)$ which is a Gaussian variable with strictly positive variance (because $\Var \Re W(1) = \sigma_1^2>0$) and hence, non-vanishing a.s.
\end{proof}

Let us check that $\P[Z_{\nu}\in A(\R)]=1$ using Lemma~\ref{lem:no_mult_zeros_general}.

\begin{lemma}
With probability $1$, there is no $t\in\R$ such that $Z_{\nu}(t)= Z_{\nu}'(t)=0$.
\end{lemma}
\begin{proof}
Recall that $Z_{\nu}$ is a random analytic function. We intend to show that $Z_{\nu}(t)$ have densities which are bounded uniformly in $t\in\R$, $\varepsilon<|t|<\varepsilon^{-1}$ for fixed $\varepsilon>0$. By Lemma~\ref{lem:no_mult_zeros_general} this implies that the process $Z_{\nu}$ almost surely does not have multiple zeros in any interval bounded away from zero. Fix $\varepsilon>0$. It is enough to show that
\begin{equation}\label{eq:char_in_L1}
\int_{\R} \left| \E \eee^{\iii a Z_{\nu}(t)}\right|  {\rm d} a \leq  C,
\end{equation}
where $C$ does not depend on $\varepsilon<|t|<\varepsilon^{-1}$. This means that the characteristic function of the random variable $Z_{\nu}(t)$ has bounded $L^1$-norm which, by Fourier inversion, implies that this random variable has Lebesgue density, say $p_t$, and
$$
p_t(x) = \frac 1 {2\pi} \left|\int_{\R} \eee^{-\iii  ax} \E \eee^{\iii a Z_{\nu}(t)} {\rm d} a\right| \leq \frac C {2\pi}, \quad x\in\R, \quad \eps < |t| < \eps^{-1}.
$$
We prove~\eqref{eq:char_in_L1}. Recall that
$$
aZ_{\nu}(t) = \int_0^{1} a \sin (tu) {\rm d} \Re L(u)  + \int_0^{1} a\cos(tu)  {\rm d} \Im L(u).
$$
By a formula for the characteristic function of such stochastic integral (see, for example, formula (6) in ~\cite{iksanov_incr_response}), we have
\begin{equation}\label{eq:char_func_z_nu}
\log \E \eee^{\iii a Z_{\nu}(t)}
=
\int_0^1 \psi (a \sin (tu), a\cos(tu)) {\rm d} u,
\end{equation}
where 
$$
\psi(x,y) = \log \E \eee^{\iii (x \Re L(1) + y \Im L(1))}, \quad x,y\in\R.
$$
The random vector $(\Re L(1), \Im L(1))$ is $\alpha$-stable. Denote its spectral measure by $\Gamma$ (which is a finite measure on the unit circle $\mathbb T =\{z\in \C\colon |z|=1\}$ that can be easily expressed in terms of $\tilde \nu$). We have, see Theorem 2.3.1 in \cite{samorodnitsky_taqqu_book},
\begin{align*}
\Re \psi(a \sin (tu), a\cos(tu))
&= - |a|^{\alpha}\int_{[0,2\pi)} |\Im \eee^{\iii (tu+\phi)}|^{\alpha} \Gamma({\rm d}\phi)\\
&=
- |a|^{\alpha}\int_{[0,2\pi)} |\sin (tu+\phi)|^{\alpha} \Gamma({\rm d}\phi).
\end{align*}
Putting this into~\eqref{eq:char_func_z_nu} we obtain
\begin{align*}
\Re \log \E \eee^{\iii a Z_{\nu}(t)}
&=
 - |a|^{\alpha}\int_{[0,2\pi)}\int_0^1  |\sin (tu+\phi)|^{\alpha} {\rm d}u\Gamma({\rm d}\phi)\\
&= - |a|^{\alpha}\int_{[0,2\pi)}t^{-1}\int_{\phi}^{t+\phi}  |\sin v|^{\alpha} {\rm d}v\Gamma({\rm d}\phi).
\end{align*}
The function $(\phi,t)\mapsto t^{-1}\int_{\phi}^{t+\phi}|\sin v|^{\alpha}{\rm d}v$ is continuous and strictly positive on the compact set $[0, 2\pi]\times [\varepsilon,\varepsilon^{-1}]$, hence attains its minimal value, say $\delta>0$. Therefore,
$$
\Re \log \E \eee^{\iii a Z_{\nu}(t)} \leq -\delta \Gamma([0,2\pi))|a|^{\alpha},\quad a\in\R,
$$
yielding~\eqref{eq:char_in_L1} and proving that there are no multiple zeros in $\R\setminus\{0\}$.

 It remains to show that $t=0$ is not a multiple zero almost surely. Note that if the spectral measure $\Gamma$ is supported by $\{0,\pi\}$, then
$\Im L(u)=0$ and hence $Z_{\nu}(0)=0$ almost surely. Otherwise, $Z_{\nu}(0)=\Im L(1)$ is non-degenerate stable random variable, hence $\P[Z_{\nu}(0)=0]=0$. So let us assume that $\Gamma$ is concentrated on $\{0,\pi\}$. We have
$$
Z_{\nu}^{\prime}(0)=\int_0^1 u{\rm d}\Re L(u)
$$
implying that $Z_{\nu}^{\prime}(0)$ has a non-degenerate stable law, hence $\P[Z^{\prime}_{\nu}(0)=0]=0$. This completes the proof of the lemma.
\end{proof}

\section{Appendix}\label{sec:appendix}


The first statement of the next lemma is closely related to the classical Weyl equidistribution theorem, see, for example, \cite{Kuipers+Niederreiter:1974}, which states that for every irrational $\alpha$,
$$
\lim_{n\to\infty} \frac{1}{n} \#\{1\leq k\leq n: (k\alpha)\, {\rm mod}\, 1 \leq y\} = y, \quad y\in [0,1].
$$
It is used in the proof of Lemma~\ref{point_processes_convergence}.
\begin{lemma}\label{weyl_uniform_lemma}
Let $f:[0,\infty)\times (\overline{\C}\setminus\{0\})\to\R^{+}$ be a continuous function with compact support.
\begin{itemize}
\item[(i)]
If $\alpha\in\R\setminus \pi\Q$, then
\begin{equation}\label{weyl_uniform_convergence_irrational}
\lim_{n\to\infty}\sup_{z\in \overline{\C}\setminus\{0\}}\left|\frac{1}{n}\sum_{k\geq 1}f(k/n,\eee^{\iii k\alpha}z)-  \int_0^{\infty}\int_{0}^{1}f(x,\eee^{2\pi \iii y}z){\rm d}y{\rm d}x\right|=0.
\end{equation}
\item[(ii)]
If $\alpha=2\pi p/q$ for $p\in\Z$ and $q\in\N$ coprime, then
\begin{equation}\label{weyl_uniform_convergence_rational}
\lim_{n\to\infty}\sup_{z\in \overline{\C}\setminus\{0\}}\left|\frac{1}{n}\sum_{k\geq 1}f(k/n,\eee^{\iii k\alpha}z)-\frac{1}{q}\sum_{k=1}^{q}\int_0^{\infty}f(x,\eee^{2\pi \iii k/q}z){\rm d}x\right|=0.
\end{equation}
\end{itemize}
\end{lemma}
\begin{proof}
Fix $a,r>0$ such that $f(x,z)=0$ if $x>a$ or $|z|<r$. Consider a family $(f_z)\subset C([0,\,a]\times [0,\,2\pi])$, where $f_z(x,y):=f(x,z\eee^{\iii y})$, $y\in[0,2\pi]$, $x\in [0,\,a]$, indexed by the complex  variable $z$ such that $|z|\geq r$. Trivially, this family is uniformly bounded. Let us check that $(f_z)$ is also equicontinuous and therefore, by the Arzel\'{a}--Ascoli theorem, precompact in $C([0,\,a]\times [0,\,2\pi])$.

Fix $\varepsilon>0$. Since $f$ is uniformly continuous on $[0,\infty) \times (\overline{\C}\setminus\{0\})$ we can find $\delta>0$ such that for every $z_1,z_2\in \overline{\C}\setminus\{0\}$, $|z_1-z_2|<\delta$ and every $x_1,x_2\in [0,\infty)$, $|x_1-x_2|<\delta$, we have $|f(x_1,z_1)-f(x_2,z_2)|<\varepsilon/3$. Further, by the uniform continuity of $f$, there exists a finite limit
$$f(x,\infty)=\lim_{|z|\to+\infty}f(x,z)$$
and moreover the convergence is uniform for $x\in [0,\,a]$. Pick $R=R(\varepsilon)>0$ such that $|f(x,z_1)-f(x,z_2)|<\varepsilon/3$ for $|z_1|>R$, $|z_2|>R$ and all $x\in[0,\,\infty)$. We have, for all $y_1,y_2\in [0,\,2\pi]$ and $x_1,x_2\in [0,\,a]$,
\begin{align*}
\sup_{|z|\geq r}|f_z(x_1,y_1)-f_z(x_2,y_2)|
&= \sup_{|z|\geq r}|f(x_1,z\eee^{\iii y_1})-f(x_2,z\eee^{\iii y_2})|\\
&\leq \sup_{r\leq |z|\leq R}|f(x_1,z\eee^{\iii y_1})-f(x_1,z\eee^{\iii y_2})|\\
&+ \sup_{|z| > R}|f(x_1,z\eee^{\iii y_1})-f(x_1,z\eee^{\iii y_2})|\\
&+\sup_{|z|\geq r}|f(x_1,z\eee^{\iii y_2})-f(x_2,z\eee^{\iii y_2})|.
\end{align*}
The first and the third summands are $<\varepsilon/3$ whenever $|y_1-y_2|<\delta/R$ and $|x_1-x_2|<\delta$, respectively. The second is $<\varepsilon/3$ by the choice of $R$. Hence,
$$
\sup_{|z|\geq r}|f_z(x_1,y_1)-f_z(x_2,y_2)|<\varepsilon
$$
if $|y_1-y_2|<\delta/R$ and $|x_1-x_2|<\delta$, yielding the equicontinuity.

To check~\eqref{weyl_uniform_convergence_irrational}, consider the measures on $[0,a]\times [0,\,2\pi]$ defined by
$$
\mu^{\prime}_n:=\frac{1}{[na]}\sum_{k\geq 1}\delta_{(k/n,(k\alpha){\rm mod }(2\pi))}
$$
and note that, for every $x\in(0,\,a]$ and $y\in [0,\,2\pi]$,
\begin{align}
\mu^{\prime}_n([0,\,x]\times [0,\,y])
&=\frac{\#\{k\in\N:k/n\leq x,(k\alpha){\rm mod}(2\pi) \leq y\}}{[na]}\label{weyl_joint1}\\
&\sim \frac{x}{a}\frac{\#\{k\leq nx:(k\alpha){\rm mod}(2\pi) \leq y\}}{nx}\ton \frac{xy}{2\pi a}\notag, \notag
\end{align}
where the last passage follows from Weyl's equidistribution theorem. Therefore, we have weak convergence of probability measures
\begin{equation}\label{weyl_joint2}
\mu_n^{\prime}\ton \mu^{\prime}:=(2\pi a)^{-1}\mathbb{LEB}_{[0,\,a]\times [0,\,2\pi]},
\end{equation}
since the distribution functions of $\mu^{\prime}_n$ converge pointwise to the distribution function of $\mu^{\prime}$.

By the Skorokhod representation theorem there exists a probability space $(\Omega,\mathcal{A},\P)$ and random vectors $X_n$ and $X$ on this space such that $X_n$ has distribution $\mu^{\prime}_n$, $X$ has distribution $\mu^{\prime}$, and $X_n\to X$, as $n\to\infty$, almost surely. With this notation we can recast~\eqref{weyl_uniform_convergence_irrational} as follows:
\begin{equation}\label{weyl_uniform_convergence2}
\lim_{n\to\infty}a\sup_{|z|\geq r}\left|\E_{\P}f_z(X_n)-\E_{\P}f_z(X)\right|=0.
\end{equation}
By the precompactness (or just equicontinuity) of the family $(f_z)$:
$$
\sup_{|z|\geq r}\left|f_z(X_n)-f_z(X)\right|\toas 0.
$$
Recalling the uniform boundedness of $(f_z)$ and invoking the dominated convergence theorem, we arrive at~\eqref{weyl_uniform_convergence2}. Relation~\eqref{weyl_uniform_convergence_rational} follows analogously from the observation that
$$
\mu_n^{\prime\prime}:=\frac{1}{na}\sum_{k\geq 1}\delta_{(k/n, (2\pi p k/q){\rm mod}(2\pi)\})}\to \mu^{\prime\prime}:=a^{-1}\mathbb{LEB}\times\mathbb{U}_q,
$$
weakly, where $\mathbb{U}_q$ is a uniform measure on the set $\{0,2\pi /q, 4\pi /q,\ldots,2\pi (q-1)/q\}$.
The proof of Lemma~\ref{weyl_uniform_lemma} is complete.
\end{proof}

The next lemma shows that the mapping $\mathcal{F}:D([0,\,1],\C)\to\mathcal{H}$ defined by~\eqref{mapping_definition} is continuous.

\begin{lemma}\label{continuous_mapping}
The mapping $\mathcal{F}$, defined by~\eqref{mapping_definition}, is everywhere continuous on $D([0,\,1],\C)$.
\end{lemma}
\begin{proof}
Let $f_n\to f$, as $n\to\infty$, on $D([0,\,1],\C)$. Fix a compact set $K\subset \C$ and let us show that
\begin{equation}\label{continuous_proof1}
\lim_{n\to\infty}\sup_{z\in K}|(\mathcal{F}(f_n))(z)-(\mathcal{F}(f))(z)|=0.
\end{equation}
By the definition of $J_1$-topology,  $f_n(1)\to f(1)$ and $f_n(0)\to f(0)$ as $n\to\infty$. Hence equation~\eqref{continuous_proof1} is equivalent to
\begin{equation}\label{continuous_proof2}
\lim_{n\to\infty}\sup_{z\in K}\left|\int_0^1 f_n(x)\eee^{\iii z x}{\rm d}x-\int_0^1 f(x)\eee^{\iii z x}{\rm d}x\right|=0.
\end{equation}
It is known that convergence in $D([0,\,1],\C)$ implies convergence in $L_1([0,\,1])$, see e.g.\ Lemma 2.2 in \cite{Johansson+Sola:2009}. Hence,
$$
\lim_{n\to\infty}\int_0^{1} |f_n(x)-f(x)|{\rm d}x=0
$$
and~\eqref{continuous_proof2} follows from the the inequalities
\begin{align*}
\sup_{z\in K}\left|\int_0^1 f_n(x)\eee^{\iii z x}{\rm d}x-\int_0^1 f(x)\eee^{\iii z x}{\rm d}x\right|
&\leq\int_0^1 \left(\sup_{z\in K}|\eee^{\iii z x}|\right)|f_n(x)-f(x)|{\rm d}x\\
&\leq\left(1+\sup_{z\in K}\eee^{-\Im z}\right)\int_0^1|f_n(x)-f(x)|{\rm d}x.
\end{align*}
The proof of Lemma~\ref{continuous_mapping} is complete.
\end{proof}

\subsection*{Acknowledgements}
This work was done while A.~Iksanov was visiting M\"{u}nster in
January--February 2016. He gratefully acknowledges hospitality and the financial support by DFG SFB 878 ``Geometry, Groups and Actions''. The work of
A.~Marynych was supported by the Alexander von Humboldt
Foundation.

\bibliographystyle{plain}
\bibliography{loc_universality_bib}

\end{document}